\documentclass[12pt,letterpaper]{amsart}
\usepackage{amsmath, amsthm, amssymb, xspace, mathrsfs}
\usepackage[tmargin=1.4in,bmargin=1.2in,rmargin=1.4in,lmargin=1.4in]{geometry}
\usepackage[breaklinks=true]{hyperref}
\usepackage{amsmath,amscd}
\usepackage{tikz-cd}

\theoremstyle{plain}
\newtheorem{theorem}{Theorem}[section]

\theoremstyle{definition}

\newtheorem{remark}{Remark}[section]
\newtheorem*{acknowledgement}{Acknowledgements}

\newcommand{\comment}[1]{}

\DeclareMathOperator{\Id}{\ensuremath{Id}}

\begin{document}
\title{Induced subgraphs of powers of oriented cycles}
\author{Akaki Tikaradze}
\email{ tikar06@gmail.com}
\address{University of Toledo, Department of Mathematics \& Statistics, 
Toledo, OH 43606, USA}
\begin{abstract}

By using a $q$-analogue of the ``magic" matrix introduced by H.Huang in his elegant solution of
the sensitivity conjecture, we give a direct generalization of his result, replacing
a hypercube graph by a Cartesian power of a { directed} $l$-cycle.

\end{abstract}

\maketitle

Very recently, Hao Huang gave a very beautiful and short proof of the sensitivity
conjecture \cite{H}. The key ingredient of his proof was the construction
of the ``magic" matrix that allows an estimate for eigenvalues
of the adjacency matrix of the hypercube graph. 
 In this note we introduce a $q$-analogue of his construction, where $q$ is a primitive
 $ l$-th root of unity, $l>1.$
This allows us to obtain an analogue of Huang's theorem
replacing hypercubes by powers of $l$-cycles.

We will fix once and for all $l>1$ and an $l$-th primitive root of unity $q$.
Recall that in an associative $\mathbb{C}$-algebra, if two elements $q$-commute: $ab=q ba,$
then $(a+b)^l=a^l+b^l.$
Let $x$ be a diagonal $l$-by-$l$ matrix whose $(i,i)$-entry is $q^i.$
Let $y$ be an $l$-by-$l$ permutation matrix representing a cycle: thus $y_{i, i+1}=1=y_{n,1}, 1\leq i<n$
and the rest of the entries of $y$ are 0. It is easy to check that $xy=q yx.$
Also, $x^l=y^l=\Id_l.$

Given an $m$-by-$m$ (complex) matrix $A$, 
by $x(A)$ (respectively $y(A)$) we denote an $ml$-by-$ml$ matrix obtained
by replacing the $(i,j)$-entry $x_{ij}$ by $x_{ij}A$ (resp. replacing $y_{ij}$ by $y_{ij}A$).
Similarly $\Id_l(A)$ denotes an $ml$-by-$ml$ matrix obtained by inserting
$A$ in all diagonal entries.  (These are simply the Kronecker
products $x \otimes A, y \otimes A, \Id_l \otimes A$.)
Clearly $x(A)y(\Id_l)=q y(\Id_l)x(A)$.
By $\tilde{A}$ we will denote $x(A)+y(\Id_m).$ Thus $\tilde{A}$ is an $ml$-by-$ml$ matrix. 
Hence $$(\tilde{A})^l=\Id_l(A^l+\Id_m).$$
Next we define inductively a sequence of matrices as follows: $B_1=y, B_n=\tilde{B}_{n-1}, n\geq 1.$
Thus we may  conclude that $B_n^l=n\Id_{l^n}.$ Hence $B_n$ is a diagonalizable $l^n$-by-$l^n$
matrix with eigenvalues satisfying the equation $x^l-n=0$ , with at least one of them having multiplicity $\geq l^{n-1}.$
Clearly all entries in $B_n$ are either 0 or $l$-th roots of unity.
When $l=2$, we have $q=-1$ and we recover the matrices constructed by Huang
\cite{H}.

Next, we will relate the above matrices with graph theory.
Given a directed graph $S$ with $m$-vertices, by $M(S)$  (the adjacency matrix of $S$) we denote the $m$-by-$m$ matrix
whose $(i, j)$-th entry is the number
of edges from $i$ to $j,$ where $i, j$ are vertices of $S.$ We assume
from now on that there is at most one edge from $i$ to $j$. Thus, the
entries of $M(S)$ consist of $0, 1$ only.

Let $C_l$ denote the oriented $l$-cycle. Let $C_l^n$ be the $n$-th Cartesian power of $C_l.$
Then it is easy to see that $M(C_l^n)$ is obtained from $B_n$ by replacing every entry with its absolute value.

Now the proof of the following result is identical to Tao's exposition of Huang's proof \cite{T}.

\begin{theorem}\label{main}

 Let $S$ be a subset of vertices of  $C_l^n$
of size $(l-1)l^{n-1}+1.$ Then either there exist a vertex in $S$ that is joined with 
at least $n^{1/l}$-many vertices in $S$, or there exist a vertex is $S$ such that
there are $n^{1/l}$-many vertices in $S$ that are joined to it.
\end{theorem}

\begin{proof}
We need to show that if we restrict $M(C_l^n)$ to its $|S|\times |S|$ minor
corresponding to vertices in $S,$ then it will have either row or a column containing
at least $n^{1/l}$-nonzero entries. It suffices to show the  analogous statement for
the matrix $B_n.$
  As remarked above, $B_n$ has an eigenvalue $\lambda$ such that $\lambda^l=n$
  and the dimension of the corresponding eigenspace is at least $l^{n-1}.$
Now viewing $B_n$ as a linear operator on the $\mathbb{C}$-span of vertices, we see that
the eigenspace of $\lambda$ must intersect nontrivially the span of $S$.
Hence the $S\times S$ minor of $B_n,$ to be denoted by $B_S$, has an eigenvalue with the absolute
value $n^{1/l}.$ Hence by Schur's inequality, $B_S$ has either a row or a column with $l^1$-norm
 at least $n^{1/l}.$ Now since all nonzero entries of $B_S$ have absolute value 1, we are done.
\end{proof}

\begin{remark}
For $l=2,$ { the `undirected' specialization of} Theorem  \ref{main} is Theorem 1.1  from \cite{H}.
While it is well-known that the lower bound $\sqrt{n}$ is tight for the case of a hypercube \cite{3}, we do not know
how far off from the optimal bound is the lower bound $n^{\frac{1}{l}}$ for $l>2.$ However, the minimal number of vertices
required for a nontrivial conclusion is not too far off from the optimal one as the following
example kindly communicated by the anonymous referee shows. 

Let us represent vertices of $C_l^m$ as $m$-tuples $a=(a_1,\cdots, a_m)$ with $a_i\in \lbrace 0, 1,\cdots, l-1\rbrace.$
So if there is an arrow $a\to b$ then  $\sum_{i=1}^m(a_i-b_i)=1 \mod l.$ Let $S_k$ consists of those vertices
$a$ such that $\sum_ia_i=k \mod l,$ so $|S_k|=l^{m-1}.$ Put $X=\bigcup_{k<\frac{l-1}{2}}S_{2k+1}.$
Then $|X|=[(l-1)/2]l^{m-1}$ and no two vertices in $X$ are connected.

\end{remark}
\begin{acknowledgement}
I am grateful to the anonymous referee for several useful suggestions especially
for providing the example given in the paper.

\end{acknowledgement}


\begin{thebibliography}{H}




\bibitem[1]{H}
H.~Huang, {\em Induced subgraphs of hypercubes and a proof of the Sensitivity Conjecture}, 
Annals of Math. Vol. 190 Issue 3 (2019) 949--955.



\bibitem[2]{T}

T.~Tao, {\em Twisted convolution and the sensitivity conjecture},
\href{https://terrytao.wordpress.com/2019/07/26/twisted-convolution-and-the-sensitivity-conjecture/}{terrytao.wordpress.com}.




\bibitem[3] {3}

F.~Chung, Z.~ Füredi, R.~ Graham,  P.~Seymour,
{\em On induced subgraphs of the cube}, J. Combin. Theory Ser. A 49 (1988), no. 1, 180--187.

\end{thebibliography}
\end{document}